\documentclass[oneside,english,british,reqno]{amsart}
\usepackage[latin9]{inputenc}
\usepackage{xcolor}
\usepackage{babel}
\usepackage{amsthm}
\usepackage{amstext}
\usepackage{amssymb}
\usepackage{setspace}
\usepackage{esint}
\PassOptionsToPackage{normalem}{ulem}
\usepackage{ulem}
\usepackage[unicode=true,pdfusetitle,
 bookmarks=true,bookmarksnumbered=false,bookmarksopen=false,
 breaklinks=false,pdfborder={0 0 0},backref=false,colorlinks=true]
 {hyperref}
\usepackage{breakurl}

\makeatletter

\providecolor{lyxadded}{rgb}{0,0,1}
\providecolor{lyxdeleted}{rgb}{1,0,0}
\newcommand{\lyxdeleted}[3]{{\texorpdfstring{\color{lyxdeleted}\sout{#3}}{}}}

\numberwithin{equation}{section}
\numberwithin{figure}{section}
\usepackage{enumitem}		% customizable list environments
\newlength{\lyxlistindent}      % custimizable indent
\theoremstyle{plain}
\newtheorem{thm}{\protect\theoremname}
  \theoremstyle{definition}
  \newtheorem{defn}[thm]{\protect\definitionname}
  \theoremstyle{plain}
  \newtheorem{lem}[thm]{\protect\lemmaname}
  \theoremstyle{remark}
  \newtheorem{rem}[thm]{\protect\remarkname}
  \theoremstyle{definition}
  \newtheorem{example}[thm]{\protect\examplename}
  \theoremstyle{plain}
  \newtheorem{cor}[thm]{\protect\corollaryname}

\usepackage{babel}
\usepackage{pdfsync}\usepackage{graphics}\usepackage{epstopdf}\usepackage[all]{xy}% per diagrammi categoriali
\usepackage{amscd}%\usepackage[inline]{enumitem}
\newcommand{\xyR}[1]{ \makeatletter
\xydef@\xymatrixrowsep@{#1} \makeatother} % end of \xyR
\newcommand{\xyC}[1]{ \makeatletter
\xydef@\xymatrixcolsep@{#1} \makeatother} % end of \xyC
\entrymodifiers={++[ ][F]} % per avere piu' spazio intorno agli oggetti nei diagrammi

\newcommand{\eps}{\varepsilon} % for the sake of brevity only
\renewcommand{\phi}{\varphi} %phi normale
\newcommand{\field}[1]{\mathbb{#1}}
\newcommand{\R}{\field{R}}                        % reals
\newcommand{\N}{\field{N}}                        % naturals
\newcommand{\diff}[1]{\,{\rm d}#1}
\newcommand{\Cc}{\mathcal{C}} % for the sake of brevity in various situations
\newcommand{\Coo}{\Cc^\infty}
\newcommand{\D}{\mathcal{D}} % generalized elements or figures
\newcommand{\then}{\quad \Longrightarrow \quad}
\newcommand{\DIff}{ \quad\;\; :\!\iff \quad } % :iff by definition
\newcommand{\gss}{{\mathcal G}^{\text{\rm s}}}
\newcommand{\gse}{{\mathcal G}^{\text{\rm e}}}
\newcommand{\gsd}{{\mathcal G}^{\text{\rm d}}}
\newcommand{\cG}{{\mathcal{G}}}
\newcommand{\erm}{\text{\rm e}}
\newcommand{\srm}{\text{\rm s}}
\newcommand{\drm}{\text{\rm d}}

\newcommand{\brm}{\text{\rm b}}

\newcommand{\mbR}{\mathbb{R}}

\newcommand{\ep}{\varepsilon}

\newcommand{\Rtil}{\widetilde \R}

\newcommand{\supp}{\mbox{supp}}

\newcommand{\gs}{{\mathcal G}}

\newcommand{\otilc}{\widetilde \Omega_c}

  \addto\captionsbritish{\renewcommand{\corollaryname}{Corollary}}
  \addto\captionsbritish{\renewcommand{\definitionname}{Definition}}
  \addto\captionsbritish{\renewcommand{\examplename}{Example}}
  \addto\captionsbritish{\renewcommand{\lemmaname}{Lemma}}
  \addto\captionsbritish{\renewcommand{\remarkname}{Remark}}
  \addto\captionsbritish{\renewcommand{\theoremname}{Theorem}}
  \addto\captionsenglish{\renewcommand{\corollaryname}{Corollary}}
  \addto\captionsenglish{\renewcommand{\definitionname}{Definition}}
  \addto\captionsenglish{\renewcommand{\examplename}{Example}}
  \addto\captionsenglish{\renewcommand{\lemmaname}{Lemma}}
  \addto\captionsenglish{\renewcommand{\remarkname}{Remark}}
  \addto\captionsenglish{\renewcommand{\theoremname}{Theorem}}
  \providecommand{\corollaryname}{Corollary}
  \providecommand{\definitionname}{Definition}
  \providecommand{\examplename}{Example}
  \providecommand{\lemmaname}{Lemma}
  \providecommand{\remarkname}{Remark}
\providecommand{\theoremname}{Theorem}

\makeatother

  \addto\captionsbritish{\renewcommand{\corollaryname}{Corollary}}
  \addto\captionsbritish{\renewcommand{\definitionname}{Definition}}
  \addto\captionsbritish{\renewcommand{\examplename}{Example}}
  \addto\captionsbritish{\renewcommand{\lemmaname}{Lemma}}
  \addto\captionsbritish{\renewcommand{\remarkname}{Remark}}
  \addto\captionsbritish{\renewcommand{\theoremname}{Theorem}}
  \addto\captionsenglish{\renewcommand{\corollaryname}{Corollary}}
  \addto\captionsenglish{\renewcommand{\definitionname}{Definition}}
  \addto\captionsenglish{\renewcommand{\examplename}{Example}}
  \addto\captionsenglish{\renewcommand{\lemmaname}{Lemma}}
  \addto\captionsenglish{\renewcommand{\remarkname}{Remark}}
  \addto\captionsenglish{\renewcommand{\theoremname}{Theorem}}
  \providecommand{\corollaryname}{Corollary}
  \providecommand{\definitionname}{Definition}
  \providecommand{\examplename}{Example}
  \providecommand{\lemmaname}{Lemma}
  \providecommand{\remarkname}{Remark}
\providecommand{\theoremname}{Theorem}

\begin{document}

\title{Unifying order structures for Colombeau algebras}

\author{Paolo Giordano \and Eduard A. Nigsch}

\thanks{P.\ Giordano has been supported by grant P25116-N25 of the Austrian
Science Fund FWF.}

\address{\textsc{Faculty of Mathematics, University of Vienna, Austria, Oskar-Morgenstern-Platz
1, 1090 Wien, Austria}}

\thanks{E.\ Nigsch has been supported by the Austrian Science Fund (FWF)
projects P23714 and P25064. }

\email{\texttt{paolo.giordano@univie.ac.at, eduard.nigsch@univie.ac.at}}


\begin{abstract}
We define a general notion of \emph{set of indices} which, using concepts
from pre-ordered sets theory, permits to unify the presentation of
several Colombeau-type algebras of nonlinear generalized functions.
In every set of indices it is possible to generalize Landau's notion
of big-O such that its usual properties continue to hold. Using this
generalized notion of big-O, these algebras can be formally defined
the same way as the special Colombeau algebra. Finally, we examine
the scope of this formalism and show its effectiveness by applying
it to the proof of the pointwise characterization in Colombeau algebras.
\end{abstract}
\maketitle

\section{Introduction}

Colombeau algebras are algebras of generalized functions introduced
by J.-F.\ Colombeau in order to rigorously define multiplication
and other nonlinear operations on Schwartz distributions in a consistent
way. Containing the space of Schwartz distributions as a linear subspace
and the algebra of smooth functions as a faithful subalgebra, they
permit to bypass the Schwartz impossibility result. We refer to \cite{Col84,Col85,Col92,GKOS,MO92}
for detailed information; our terminology and notation mainly follows
\cite{GKOS}. Besides Colombeau's original algebra, the full algebra
$\gse$ and the special algebra $\gss$ on open subsets of $\R^{n}$
appeared (\cite{Col84,Col85,Col92}) and some years later the diffeomorphism
invariant local algebra $\gsd$ (\cite{GKOS}) was constructed.

A parallel thread, using nonstandard Analysis (NSA) methods, arrived
at a similar algebra $\hat{\mathcal{G}}$, (called \emph{algebra of
asymptotic functions}, see e.g.\ \cite{ToVe08} and references therein)
that has better formal properties: the scalars of the algebra form
an algebraically closed Cantor complete field, it is defined using
a reduced number of quantifiers, and for it a Hahn-Banach extension
principle holds (\cite{ToVe08}).

Because there are many variants of Colombeau algebras in use today,
it is desirable to gain a better understanding of their common structure
as well as their distinguishing properties. In the present work, we
will examine in which way suitable notions from the theory of pre-ordered
sets permit to unify the formal presentation of some of these algebras.
In particular, we will introduce the notion of \emph{set of indices},
which allows for a generalization of Landau's asymptotic relations
preserving their formal properties. Using these new generalized asymptotic
relations, we will reformulate the definitions of the algebras $\gss$,
$\hat{\mathcal{G}}$, $\gse$ and $\gsd$ mentioned above using the
same reduced number of quantifiers of the special one.

We start by introducing new notations for the mollifier operator $S_{\eps}$
and for the translation operator $T_{x}$ (cf.\ \cite[Section 2.3.2]{GKOS})
in order to emphasize that they are group actions on the space $\D(\mbR^{n})$
of test functions on $\mbR^{n}$. We include zero in the natural numbers
$\N=\{0,1,2,\ldots\}$.
\begin{defn}
\label{def:operators}For $\phi\in\D(\R^{n})$, $r\in\R_{>0}$ and
$x\in\R^{n}$ we define 
\begin{enumerate}[%
leftmargin=*,label=(\roman*),align=left %
]
\item \foreignlanguage{british}{$r\odot\phi:y\in\R^{n}\mapsto S_{r}\phi(y):=\frac{1}{r^{n}}\cdot\phi\left(\frac{y}{r}\right)\in\R$,} 
\item $x\oplus\phi:y\in\R^{n}\mapsto T_{x}\phi(y):=\phi(y-x)\in\R$.
\end{enumerate}
\end{defn}
\noindent It is easy to prove that $\odot$ is an action of the multiplicative
group $(\R_{>0},\cdot,1)$ on $\D(\R^{n})$ and $\oplus$ is an action
of the additive group $(\R^{n},+,0)$ on $\D(\R^{n})$. Moreover,
$r\odot(x\oplus\phi)=rx\oplus r\odot\phi$ for $r\in\R_{>0}$, $x\in\R^{n}$
and $\varphi\in\D(\R^{n})$. The following lemma will be used later.
\begin{lem}
\label{lem:identityForOperators} Let $\phi\in\D(\R^{n})\setminus\{0\}$,
$r\in\R_{>0}$ and $x\in\R^{n}$, then the actions $\odot$ and $\oplus$
are free, i.e.: 
\begin{enumerate}[%
leftmargin=*,label=(\roman*),align=left %
]
\item \foreignlanguage{british}{\label{enu:indetOdot}$r\odot\phi=\phi$
if and only if $r=1$} 
\item \label{enu:identOplus}$x\oplus\phi=\phi$ if and only if $x=0$
\end{enumerate}
\end{lem}
\begin{proof}
\ref{enu:indetOdot} $\Rightarrow$: The equality $r\odot\phi=\phi$
means $\frac{1}{r^{n}}\phi\left(\frac{x}{r}\right)=\phi(x)$ for each
$x\in\R^{n}$, which directly implies 
\begin{equation}
\frac{1}{r}\cdot\text{supp}(\phi)=\text{supp}(\phi).\label{eq:1Over_r_supp}
\end{equation}
 The support set $\text{supp}(\phi)$ is closed, bounded and non empty
because $\phi\ne0$. Take $x\in\text{supp}(\phi)$ such that $|x|$
is maximum, from \eqref{eq:1Over_r_supp} we get $\frac{x}{r}\in\text{supp}(\phi)$
and hence 
\begin{equation}
\left|\frac{x}{r}\right|=\frac{1}{r}|x|\le|x|.\label{eq:xOver_rle_x}
\end{equation}
 $|x|=0$ would imply $\text{supp}(\phi)=\{0\}$, which is impossible
since $\phi$ is continuous. Therefore, \eqref{eq:xOver_rle_x} implies
$\frac{1}{r}\le1$. But $r\odot\phi=\phi$ implies $\phi=\frac{1}{r}\odot\phi$,
hence, with the same reasoning, we also get $r\le1$, from which the
conclusion follows.

\noindent \ref{enu:identOplus} $\Rightarrow$: The equality $x\oplus\phi=\phi$
means $\phi(\cdot-x)=\phi(\cdot)$ and hence 
\begin{equation}
\text{supp}(\phi)-x=\text{supp}(\phi).\label{eq:supp-x}
\end{equation}
 There exists $y_{0}\in\text{supp}(\phi)$, and $L:=\left\{ y_{0}-tx\mid t\ge0\right\} $
is closed, so $K:=L\cap\text{supp}(\phi)$ is compact. Therefore,
there exists a point $y\in K$ where the distance $|y-y_{0}|$ is
maximum. We can write $y=y_{0}-tx$ for some $t\ge0$ because $y\in L$.
By \eqref{eq:supp-x} we also get $y-x\in\text{supp}(\phi)$. But
$y-x=y_{0}-(t+1)x\in L$, so $y-x\in K$ and thus $|y-y_{0}|\ge|y-x-y_{0}|$,
i.e.\ $t|x|\ge(t+1)|x|$, which implies $|x|=0$.
\end{proof}

\section{Set of indices}

The formulation of Colombeau algebras always involves asymptotic estimates.
The basic idea of the following definitions is to clarify and abstract
these asymptotics and thus to unify the notations and the logical
structure of Colombeau algebras.
\begin{defn}
\label{def:setOfIndices}We say that $\mathbb{I}=(I,\le,\mathcal{I})$
is a \emph{set of indices} if the following conditions hold: 
\begin{enumerate}[%
leftmargin=*,label=(\roman*),align=left %
]
\item \foreignlanguage{british}{\label{enu:DefSoI-preorder}$(I,\le)$
is a pre-ordered set, i.e., $I$ is a non empty set with a reflexive
and transitive relation $\le$.} 
\item \label{enu:DefSoI-union}$\mathcal{I}$ is a set of subsets of $I$
such that $\emptyset\notin\mathcal{I}$ and $I\in\mathcal{I}$.
\item \label{enu:DefSoI-intersection}$\forall A,B\in\mathcal{I}\,\exists C\in\mathcal{I}:\ C\subseteq A\cap B$.
\end{enumerate}

\noindent For all $e\in I$, set $(\emptyset,e]:=\left\{ \eps\in I\mid\eps\le e\right\} $.
As usual, we say $\eps<e$ if $\eps\le e$ and $\eps\ne e$. Using
these notations, we state the last condition in the definition of
set of indices:
\begin{enumerate}[leftmargin={*},label=(\roman*),align=left,start=4]
\item \label{enu:DefSoI-DonwDir}If $e\le a\in A\in\mathcal{I}$, the set
$A_{\le e}:=(\emptyset,e]\cap A$ is downward directed by $<$, i.e.,
it is non empty and 
\begin{equation}
\forall b,c\in A_{\le e}\,\exists d\in A_{\le e}:\ d<b\ ,\ d<c.\label{eq:strictlyDownDirected}
\end{equation}
 
\end{enumerate}
\end{defn}
\begin{rem}
\label{rem:firstConsequencesDefSoI}\ 
\begin{enumerate}[%
leftmargin=*,label=(\roman*),align=left %
]
\item Conditions \ref{enu:DefSoI-union} and \ref{enu:DefSoI-intersection}
can be summarized saying that $\mathcal{I}$ is a filter base on $I$
which contains $I$.
\item Let us note explicitly that in \ref{enu:DefSoI-DonwDir} it is not
required that $e\in A$. In Thm. \ref{thm:usualBigOh-prop} \ref{enu:1stBigOh-subsetA}
we will motivate this choice.
\item Since $I\in\mathcal{I}$, condition \ref{enu:DefSoI-DonwDir} yields
that $(\emptyset,e]$ is downward directed by $<$ for each $e\in I$. 
\item In the set of indices that we will define for the full algebra $\gs^{\erm}(\Omega)$
(see Def.\ \ref{def:setOfIndicesFullAlgebra} below), we will see
that in general $(\emptyset,e]$ is not an element of $\mathcal{I}$.
In the same example, we have that in general $A\in\mathcal{I}$ is
not downward directed. 
\end{enumerate}
\end{rem}
In order to illustrate this definition we will give some examples.
\begin{example}
\label{exa:stdAndNS}\ 
\begin{enumerate}[%
leftmargin=*,label=(\roman*),align=left %
]
\item \foreignlanguage{british}{\label{enu:I^s}The simplest example of
set of indices is given by $I^{\srm}:=(0,1]\subseteq\R$, the relation
$\le$ is the usual order relation on $\R$, and $\mathcal{I}^{\srm}:=\left\{ (0,\eps_{0}]\mid\eps_{0}\in I\right\} $.
We denote by $\mathbb{I}^{\srm}:=(I^{\srm},\le,\mathcal{I}^{\srm})$
this set of indices which, as we will see, is the one used for the
special algebra $\gss$. }
\item Let $\mathcal{I}$ be an ultrafilter on $\N$ containing the Fréchet
filter (\cite{Gol98}). Let $\ge$ be the usual order relation on
the natural numbers. Then $(\N,\ge,\mathcal{I})$ is a set of indices
which can also be used for the formulation of the special algebra. 
\item \label{enu:Ihat}In the context of \cite{ToVe08}, we set $\hat{I}:=\D_{0}=\D(\R^{d})$.
The pre-order relation is defined by $\phi\le\psi$ iff $\underline{\phi}\le\underline{\psi}$,
where $\underline{\phi}:=\text{diam}\left(\text{supp}(\phi)\right)$
(the diameter of the support of $\varphi$) if $\phi\ne0$ and $\underline{\phi}:=1$
otherwise. Note that this is only a pre-order and not an order relation.
$\hat{\mathcal{I}}$ is the free ultrafilter on $\D_{0}$ employed
in \cite{ToVe08}, and we set $\hat{\mathbb{I}}:=(\hat{I},\le,\hat{\mathcal{I}}$).
In the following Thm.\ \ref{thm:IhatIsSetOfIndices} it is proved
that this is actually a set of indices. 
\end{enumerate}
\end{example}
\begin{thm}
\label{thm:IhatIsSetOfIndices}$\hat{\mathbb{I}}$ is a set of indices.\end{thm}
\begin{proof}
Properties \ref{enu:DefSoI-preorder}, \ref{enu:DefSoI-union} and
\ref{enu:DefSoI-intersection} are clear. Def. 2.1 and Thm. 2.3 in
\cite{ToVe08} imply that we have a sequence $\left(\D_{n}\right)_{n\in\N}$
of $\hat{\mathcal{I}}$ such that 
\begin{equation}
\forall n\in\N_{>0}\,\forall\phi\in\D_{n}:\ \underline{\phi}\le\frac{1}{n}.\label{eq:D_n-R_phi}
\end{equation}
 So, if $e\le a\in A\in\hat{\mathcal{I}}$ and $\phi$, $\psi\in A_{\le e}$,
then $0<\underline{\phi}\le\underline{e}$ and $0<\underline{\psi}\le\underline{e}$.
Therefore there exists $n\in\N_{>0}$ such that $\frac{1}{n}<\min\left(\underline{\phi},\underline{\psi}\right)\le\underline{e}$.
Since $\D_{n}\in\mathcal{I}$, also $\D_{n}\cap A\in\hat{\mathcal{I}}$
from \ref{enu:DefSoI-intersection} of Def. \ref{def:setOfIndices}.
But $\hat{\mathcal{I}}$ is an ultrafilter, so there exists $d\in\D_{n}\cap A$.
Applying \eqref{eq:D_n-R_phi} with $\phi=d$ we obtain $\underline{d}\le\frac{1}{n}<\min\left(\underline{\phi},\underline{\psi}\right)\le\underline{e}$
which is the conclusion.
\end{proof}

\subsection{Two notions of big-O in a set of indices}

In each set of indices, we can define two notions of big-O that formally
behave in the usual way.

Since each set of the form $A_{\le a}=(\emptyset,a]\cap A$ is downward
directed, the first big-O is the usual one: 
\begin{defn}
\label{def:usualBigOh}Let $\mathbb{I}=(I,\le,\mathcal{I})$ be a
set of indices. Let $a\in A\in\mathcal{I}$ and $(x_{\eps})$, $(y_{\eps})\in\R^{I}$
be two nets of real numbers defined in $I$. We write 
\begin{equation}
x_{\eps}=O_{a,A}(y_{\eps})\ \text{as }\eps\in\mathbb{I}\label{eq:usualBigOh}
\end{equation}
 if 
\begin{equation}
\exists H\in\R_{>0}\,\exists\eps_{0}\in A_{\le a}\,\forall\eps\in A_{\le\eps_{0}}:\ |x_{\eps}|\le H\cdot|y_{\eps}|.\label{eq:1stBigOhDef}
\end{equation}
We explicitly note that the variable $\eps$ in \eqref{eq:usualBigOh}
is actually a mute variable. As usual (see e.g.\ \cite{Bou}), the
notation $x_{\eps}=O_{a,A}(y_{\eps})$ really represents a pre-order
relation, and the use of the equality sign is an abuse of language.
From this point of view, a Vinogradov notation like $x_{\eps}\ll_{a,A}y_{\eps}$
would surely be better. Another innocuous abuse of language is the
use of the symbol $\le$ for the pre-order relation on $I$ and for
the order relation on the reals used in the last part of \eqref{eq:1stBigOhDef}.\end{defn}
\begin{example}
\label{exa:usualBigOh}\ 
\begin{enumerate}[%
leftmargin=*,label=(\roman*),align=left %
]
\item \foreignlanguage{british}{In the set of indices $\mathbb{I}^{\srm}$
of the special algebra (see \ref{enu:I^s} of Ex. \ref{exa:stdAndNS}),
the following are equivalent:}

\begin{enumerate}
\item $\forall A\in\mathcal{I}^{\srm}\,\forall a\in A:\ x_{\eps}=O_{a,A}(y_{\eps})$
as $\eps\in\mathbb{I}^{\srm}$ (or any other combination of quantifiers
$\exists A\,\exists a$, $\forall A\,\exists a$, $\exists A\,\forall a$) 
\item $x_{\eps}=O(y_{\eps})$ as $\eps\to0^{+}$ 
\end{enumerate}
\item In the set of indices $\hat{\mathbb{I}}$ of the algebra of asymptotic
functions $\hat{\gs}$ (\ref{enu:Ihat} of Ex. \ref{exa:stdAndNS}),
the following are equivalent:

\begin{enumerate}
\item \label{enu:usualBigOh1}$\exists A\in\mathcal{\hat{I}}\,\forall a\in A:\ x_{\eps}=O_{a,A}(y_{\eps})$
as $\eps\in\hat{\mathbb{I}}$ 
\item \label{enu:usualBigOh2}$\exists H\in\R_{\ne0}:\ |x_{\phi}|\le H\cdot|y_{\phi}|$
almost everywhere, where a property $\mathcal{P}(\phi)$ is said to
hold \emph{almost everywhere} iff $\left\{ \phi\in\D_{0}\mid\mathcal{P}(\phi)\right\} \in\mathcal{I}$
(see \cite{ToVe08}). 
\end{enumerate}

\noindent To prove this equivalence, we need the following
\begin{lem}
\noindent \label{lem:NSA-allIntervals}In the set of indices $\hat{\mathbb{I}}$,
we have $(\emptyset,\eps_{o}]\in\hat{\mathcal{I}}$ for all $\eps_{0}\in\hat{I}$.\end{lem}
\begin{proof}
\noindent Since $\underline{\eps_{0}}>0$, for $n\in\N_{\ne0}$ sufficiently
big we have $\frac{1}{n}\le\underline{\eps_{0}}$. From \eqref{eq:D_n-R_phi}
we thus have $\underline{\phi}\le\frac{1}{n}\le\underline{\eps_{0}}$
for each $\phi\in\D_{n}$. Therefore $\D_{n}\subseteq(\emptyset,\eps_{0}]$.
But $\D_{n}\in\hat{\mathcal{I}}$ and $\hat{\mathcal{I}}$ is an ultrafilter,
so also $(\emptyset,\eps_{0}]\in\hat{\mathcal{I}}$.
\end{proof}

\noindent Now, we prove that the previous \ref{enu:usualBigOh1} and
\ref{enu:usualBigOh2} are equivalent.

\ref{enu:usualBigOh1} $\Rightarrow$ \ref{enu:usualBigOh2}: Property
\ref{enu:usualBigOh1} yields 
\begin{equation}
\forall a\in A\,\exists H\in\R_{>0}\,\exists\eps_{0}\in A_{\le a}\,\forall\eps\in A_{\le\eps_{0}}:\ |x_{\eps}|\le H\cdot|y_{\eps}|.\label{eq:1InhatI}
\end{equation}
But there always exists $a\in A$ because $A$ is an ultrafilter set,
so $|x_{\eps}|\le H\cdot|y_{\eps}|$ for all $\eps\in A_{\le\eps_{0}}$.
Lem. \ref{lem:NSA-allIntervals} yields $(\emptyset,\eps_{0}]\in\hat{\mathcal{I}}$
and hence $A_{\le\eps_{0}}=(\emptyset,\eps_{0}]\cap A\in\hat{\mathcal{I}}$.
Hence, we can say that 
\[
\left\{ \eps\in\hat{I}\mid|x_{\eps}|\le H\cdot|y_{\eps}|\right\} \supseteq A_{\le\eps_{0}}\in\hat{\mathcal{I}}.
\]
 The conclusion follows because $\hat{\mathcal{I}}$ is an ultrafilter.

\ref{enu:usualBigOh2} $\Rightarrow$ \ref{enu:usualBigOh1}: Property
\ref{enu:usualBigOh2} means $A:=\left\{ \eps\in\hat{I}\mid|x_{\eps}|\le H\cdot|y_{\eps}|\right\} \in\hat{\mathcal{I}}$.
For each $a\in A$ we set $\eps_{0}:=a$ so that \eqref{eq:1InhatI}
follows by definition of $A$.

\end{enumerate}
\end{example}
\begin{thm}
\label{thm:usualBigOh-prop}Let $\mathbb{I}=(I,\le,\mathcal{I})$
be a set of indices, $e\in A\in\mathcal{I}$ and $(x_{\eps})$, $(y_{\eps})$,
$(z_{\eps})\in\R^{I}$, then, as $\eps\in\mathbb{I}$, the following
properties of $O_{a,A}$ hold:
\begin{enumerate}[%
leftmargin=*,label=(\roman*),align=left %
]
\item \foreignlanguage{british}{\label{enu:1stBigOh-rifl}$x_{\eps}=O_{a,A}(x_{\eps})$}
\item \label{enu:1stBigOh-trans}$x_{\eps}=O_{a,A}(y_{\eps})$ and $y_{\eps}=O_{a,A}(z_{\eps})$,
then $x_{\eps}=O_{a,A}(z_{\eps})$ 
\item \label{enu:1stBigOh-prod}$O_{a,A}(x_{\eps})\cdot O_{a,A}(y_{\eps})=O_{a,A}(x_{\eps}\cdot y_{\eps})$ 
\item \label{enu:1stBigOh-sum}$O_{a,A}(x_{\eps})+O_{a,A}(y_{\eps})=O_{a,A}\left(\left|x_{\eps}\right|+\left|y_{\eps}\right|\right)$ 
\item \label{enu:1stBigOh-prodExt}$x_{\ep}\cdot O_{a,A}(y_{\eps})=O_{a,A}(x_{\eps}\cdot y_{\eps})$ 
\item \label{enu:1stBigOh-sumEqualSummand}$O_{a,A}(x_{\eps})+O_{a,A}(x_{\eps})=O_{a,A}(x_{\eps})$ 
\item \label{enu:1stBigOh-sumExt}If $x_{\eps},y_{\eps}\ge0$ for all $\eps\in I$,
then $x_{\eps}+O_{a,A}(y_{\eps})=O_{a,A}(x_{\eps}+y_{\eps})$ 
\item \label{enu:1stBigOh-prodScal1}$\forall k\in\R:\ O_{a,A}(k\cdot x_{\eps})=O_{a,A}(x_{\eps})$ 
\item \label{enu:1stBigOh-prodScal2}$\forall k\in\R:\ k\cdot O_{a,A}(x_{\eps})=O_{a,A}(x_{\eps})$ 
\item \label{enu:1stBigOh-subsetA}If $x_{\eps}=O_{a,A}(y_{\eps})$ and
$a\in B\subseteq A$, where $B\in\mathcal{I}$, then $x_{\eps}=O_{a,B}(y_{\eps})$
\end{enumerate}
\end{thm}
\begin{proof}
All properties \ref{enu:1stBigOh-rifl} - \ref{enu:1stBigOh-prodScal2}
have a similar schema of proof. Therefore, we give as example the
proof of \ref{enu:1stBigOh-prod}. As it is customary, this has to
be read as 
\[
x'_{\eps}=O_{a,A}(x_{\eps})\ ,\ y'_{\eps}=O_{a,A}(y_{\eps})\ \Rightarrow\ x'_{\eps}\cdot y'_{\eps}=O_{a,A}(x_{\eps}\cdot y_{\eps}).
\]
The assumptions of this implication yield the existence of $H$, $K$,
$\eps_{0}$, $\eps_{1}$ such that 
\begin{align}
H & >0\ ,\ \eps_{0}\in A_{\le a}\ ,\ \forall\eps\in A_{\le\eps_{0}}:\ |x'_{\eps}|\le H\cdot|x_{\eps}|\label{eq:1stFact}\\
K & >0\ ,\ \eps_{1}\in A_{\le a}\ ,\ \forall\eps\in A_{\le\eps_{1}}:\ |y'_{\eps}|\le K\cdot|y_{\eps}|.\label{eq:2ndFact}
\end{align}
Thus, there exists $\eps_{2}\in A_{\le a}$ such that $\eps_{2}<\eps_{0},\eps_{1}$,
and for each $\eps\in A_{\le\eps_{2}}$, \eqref{eq:1stFact} and \eqref{eq:2ndFact}
imply the conclusion $|x'_{\eps}\cdot y'_{\eps}|\le H\cdot K\cdot|x_{\eps}\cdot y_{\eps}|$.

In proving \ref{enu:1stBigOh-subsetA}, we need to use a peculiar
part of Def.\ \ref{def:setOfIndices}\ \ref{enu:DefSoI-DonwDir}.
Assume that \eqref{eq:1stBigOhDef} holds and $a\in B\subseteq A$,
with $B\in\mathcal{I}$. Then we have $\eps_{0}\le a\in B\in\mathcal{I}$
(note that not necessarily $\eps_{0}\in B$). By Def.\ \ref{def:setOfIndices}\ \ref{enu:DefSoI-DonwDir}
the set $B_{\le\eps_{0}}$ is directed, so it is non empty. Let $\eps_{1}\in B_{\le\eps_{0}}$.
Then for each $\eps\in B_{\le\eps_{1}}$ we have $\eps\in A_{\le\eps_{0}}$
and the conclusion follows.
\end{proof}
Frequently, claims involving Landau big-O asymptotic relations $x_{\eps}=O(y_{\eps})$
are proved by contradiction. A method frequently used in this type
of proofs concerns the existence of a decreasing sequence $(\eps_{k})_{k\in\N}$
which tends to zero and along which the net $(x_{\eps})$ is not bounded
by $(y_{\eps})$. We want to show that this method holds with great
generality in every set of indices. We start by defining in general
what the sentence ``a sequence tends to the empty set in the directed
set $A_{\le a}$'' means:
\begin{defn}
\label{def:sequenceTendsToEmptyset}\foreignlanguage{english}{Let
$\mathbb{I}=(I,\le,\mathcal{I})$ be a set of indices. Let $a\in A\in\mathcal{I}$
and $(z_{k})_{k\in\N}$ be a sequence in $A_{\le a}$. Then we say
that 
\[
(z_{k})_{k\in\N}\to\emptyset\text{ in }A_{\le a}
\]
 if 
\begin{equation}
\forall\eps_{0}\in A_{\le a}\,\exists K\in\N\,\forall k\in\N_{\ge K}:\ z_{k}<\eps_{0}.\label{eq:seqTendsToEmptysetDef}
\end{equation}
}\end{defn}
\begin{example}
\label{exa:seqTendToEmptyset}\ 
\begin{enumerate}[%
leftmargin=*,label=(\roman*),align=left %
]
\item \foreignlanguage{british}{In the set of indices $\mathbb{I}^{\srm}$,
we have that $(z_{k})_{k\in\N}\to\emptyset$ in $A_{\le a}$ if and
only if $\lim_{k\to+\infty}z_{k}=0^{+}$. }
\item In the set of indices $\hat{\mathbb{I}}$, we have that $(z_{k})_{k\in\N}\to\emptyset$
in $A_{\le a}$ if and only if $\lim_{k\to+\infty}\underline{z_{k}}=0^{+}$.
In fact, if $r\in\R_{>0}$, take $n\in\N_{\ne0}$ such that 
\begin{equation}
\frac{1}{n}<\min(\underline{a},r).\label{eq:1overn}
\end{equation}
 But $\D_{n}\cap A$ is non empty because it is an ultrafilter set
in $\hat{\mathcal{I}}$. Thus, there exists $\eps_{0}\in\D_{n}\cap A$
and $\eps_{0}\le a$ from \eqref{eq:1overn}. From \eqref{eq:seqTendsToEmptysetDef}
we get $z_{k}<\eps_{0}$ for $k$ sufficiently big, and thus $\underline{z_{k}}\le\underline{\eps_{0}}<r$.
This proves that necessarily $\left(\underline{z_{k}}\right)_{k\in\N}\to0^{+}$.
Vice versa, if $\lim_{k\to+\infty}z_{k}=0^{+}$, then for $k$ sufficiently
big we have $\underline{z_{k}}<\underline{\eps_{0}}$, i.e.\ $z_{k}<\eps_{0}$. 
\end{enumerate}
\end{example}
\begin{lem}
\label{lem:fromSeqToDecresSeq}In the hypothesis of Def. \ref{def:sequenceTendsToEmptyset},
assume that $(z_{k})_{k\in\N}\to\emptyset$ in $A_{\le a}$ and that
\begin{equation}
\forall b,c\in A_{\le a}:\ b<c\text{ or }c\le b.\label{eq:total}
\end{equation}
Then there exists a strictly decreasing subsequence $(z_{\sigma_{k}})_{k\in\N}$
of $A_{\le a}$ which tends to $\emptyset$ in $A_{\le a}$.\end{lem}
\begin{proof}
Taking $\eps_{0}=z_{0}$ in \eqref{eq:seqTendsToEmptysetDef} we get
the existence of $\sigma_{0}:=\min\left\{ k\in\N\mid z_{k}<z_{0}\right\} $,
and $z_{\sigma_{0}}<z_{0}$. If $k<\sigma_{0}$, $z_{k}<z_{\sigma_{0}}$
cannot hold and hence $z_{\sigma_{0}}\le z_{k}$ by \eqref{eq:total}.
Setting $\eps_{0}=z_{\sigma_{0}}$ in \eqref{eq:seqTendsToEmptysetDef}
we obtain the existence of $\sigma_{1}:=\min\left\{ k\in\N\mid z_{k}<z_{\sigma_{0}}\right\} $.
As before, $z_{\sigma_{1}}<z_{\sigma_{0}}$, $\sigma_{1}>\sigma_{0}$
and $z_{\sigma_{1}}\le z_{k}$ if $k<\sigma_{1}$. Continuing in this
way we can define a strictly increasing sequence $(\sigma_{k})_{k}$
such that $\left(z_{\sigma_{k}}\right)_{k}$ is strictly decreasing.
Moreover, $z_{\sigma_{n}}\le z_{k}$ whenever $\sigma_{n}>k$. This
subsequence tends to $\emptyset$ because if $\eps_{0}\in A_{\le a}$
and $z_{K}<\eps_{0}$, then for $n$ sufficiently big $\sigma_{n}>K$
and hence $z_{\sigma_{n}}\le z_{K}<\eps_{0}$.\end{proof}
\begin{thm}
\label{thm:neg-1stBigOh}\foreignlanguage{english}{Let $\mathbb{I}=(I,\le,\mathcal{I})$
be a set of indices. Let $a\in A\in\mathcal{I}$ and $(x_{\eps})$,
$(y_{\eps})\in\R^{I}$ be two nets of real numbers defined in $I$.
Assume also that $(z_{k})_{k\in\N}$ is a sequence of $A_{\le a}$
such that $(z_{k})_{k\in\N}\to\emptyset$ in $A_{\le a}$. Then the
following are equivalent:}
\begin{enumerate}[%
leftmargin=*,label=(\roman*),align=left %
]
\item \foreignlanguage{british}{\label{enu:neg-1stBigOh1}$\neg\left[x_{\eps}=O_{a,A}(y_{\eps})\text{ as }\eps\in\mathbb{I}\right]$ }
\item \label{enu:neg-1stBigOh2}For each $H\in\R_{>0}$ there exists a sequence
$(\eps_{k})_{k\in\N}$ of $A_{\le a}$ such that:

\begin{enumerate}[%
leftmargin=*,align=left %
]
\item \foreignlanguage{british}{$(\eps_{k})_{k\in\N}\to\emptyset$ }in
$A_{\le a}$
\item \label{enu:2aneg-1stBigOh}$\forall k\in\N:\ \left|x_{\eps_{k}}\right|>H\cdot\left|y_{\eps_{k}}\right|$ 
\end{enumerate}
\end{enumerate}
\end{thm}
\begin{proof}
\ref{enu:neg-1stBigOh1} $\Rightarrow$ \ref{enu:neg-1stBigOh2}:
We assume that 
\begin{equation}
\forall H\in\R_{>0}\,\forall\bar{\eps}_{0}\in A_{\le a}\,\exists\eps\in A_{\le\bar{\eps}_{0}}:\ |x_{\eps}|>H\cdot|y_{\eps}|.\label{eq:neg}
\end{equation}
Consider an $H\in\R_{>0}$. The sequence $(\eps_{k})_{k\in\N}$ is
defined recursively and using the axiom of countable choice. The first
step of the sequence is defined as follows. Since $z_{0}\in A_{\le a}$,
we can set $\bar{\eps}_{0}=z_{0}$ in \eqref{eq:neg} to obtain the
existence of $\eps_{0}\in A_{\le z_{0}}$ such that $\left|x_{\eps_{0}}\right|>H\cdot\left|y_{\eps_{0}}\right|$.
Assume that we have already proved the existence of $\eps_{k}\in A_{\le a}$
such that $\eps_{k}\le z_{k}$ and $\left|x_{\eps_{k}}\right|>H\cdot\left|y_{\eps_{k}}\right|$.
We can now apply \eqref{eq:neg} with $\bar{\eps}_{0}=z_{k+1}$ to
obtain the existence of $\eps_{k+1}\in A_{\le z_{k+1}}$ such that
$\left|x_{\eps_{k+1}}\right|>H\cdot\left|y_{\eps_{k+1}}\right|$.
Since $\eps_{k}\le z_{k}$ and $(z_{k})_{k\in\N}\to\emptyset$ in
$A_{\le a}$, also $(\eps_{k})_{k\in\N}\to\emptyset$ in $A_{\le a}$.

\ref{enu:neg-1stBigOh2} $\Rightarrow$ \ref{enu:neg-1stBigOh1}:
By contradiction, assume that\foreignlanguage{english}{ 
\begin{equation}
\exists H\in\R_{>0}\,\exists\bar{\eps}_{0}\in A_{\le a}\,\forall\eps\in A_{\le\bar{\eps}_{0}}:\ |x_{\eps}|\le H\cdot|y_{\eps}|.\label{eq:neg-1stBigOh-contr}
\end{equation}
Since $(\eps_{k})_{k\in\N}\to\emptyset$ in $A_{\le a}$, for $k$
sufficiently big we have $\eps_{k}\le\bar{\eps}_{0}$ and so $|x_{\eps_{k}}|\le H\cdot|y_{\eps_{k}}|$
by \eqref{eq:neg-1stBigOh-contr}, which contradicts \ref{enu:2aneg-1stBigOh}.}
\end{proof}
Condition \ref{enu:usualBigOh1} of Ex. \ref{exa:usualBigOh} and
the definition of the full algebra $\gse$ (see Section \ref{sec:The-full-algebra})
are the motivations for the following second notion of big-O in a
set of indices:
\begin{defn}
\label{def:2ndBigOh}Let $\mathbb{I}=(I,\le,\mathcal{I})$ be a set
of indices. Let $\mathcal{J}\subseteq\mathcal{I}$ be a non empty
subset of $\mathcal{I}$ such that 
\begin{equation}
\forall A,B\in\mathcal{J}\,\exists C\in\mathcal{J}:\ C\subseteq A\cap B.\label{eq:hypJ}
\end{equation}
 Finally, let $(x_{\eps})$, $(y_{\eps})\in\R^{I}$ be nets of real
numbers. Then we say 
\[
x_{\eps}=O_{\mathcal{J}}(y_{\eps})\text{ as }\eps\in\mathbb{I}
\]
 if 
\[
\exists A\in\mathcal{J}\,\forall a\in A:\ x_{\eps}=O_{a,A}(y_{\eps}).
\]
 We simply write $x_{\eps}=O(y_{\eps})$ (as $\eps\in\mathbb{I}$)
when $\mathcal{J=I}$, i.e.\ for $x_{\eps}=O_{\mathcal{I}}(y_{\eps})$.
\end{defn}

\subsection*{Intuitive interpretation}

In this section, we want to give an intuitive interpretation of the
structures we are introducing.

We can think of each $e\in I$ as a measuring instrument to evaluate
our observables $(x_{\eps})\in\R^{I}$. For example, we can think
of a thermometer used to measure the temperature at some point. Each
$A\in\mathcal{I}$ is a class of instruments ($A\subseteq I$) having
at least a certain \emph{accuracy}. The relation $e\le\eps$ is interpreted
as ``the measuring instrument $e$ is spatially more accurate than
$\eps$'', in the sense that every physical measuring instrument
averages the measure of an observable in a neighbourhood of some spatial
point. For the instrument $e$ this neighbourhood is smaller than
that of $\eps$. Therefore, $x_{\eps}=O_{a,A}(y_{\eps})$ can be interpreted
saying: ``We are able to state that $(x_{\eps})$ is bounded by $(y_{\eps})$
if we use any instrument $\eps$ of class $A\in\mathcal{I}$, and
whose accuracy is greater than that of $a$''. Finally, $x_{\eps}=O(y_{\eps})$
can be intuitively interpreted saying: ``We can find an accuracy
class $A\in\mathcal{I}$ such that for each instrument $a\in A$ of
that class, we can state that $x_{\eps}=O_{a,A}(y_{\eps})$''. Condition
\ref{enu:DefSoI-intersection} of Def. \ref{def:setOfIndices} states
that from two accuracy classes $A,B\in\mathcal{I}$ we can always
find an accuracy class $C\in\mathcal{I}$ such that $C\subseteq A\cap B$,
i.e.\ whose instruments have accuracy greater or equal to that of
both $A$ and $B$. Condition \ref{enu:DefSoI-DonwDir} of Def. \ref{def:setOfIndices}
states that taking the instrument $e\le a\in A\in\mathcal{I}$, we
can always take instruments which are spatially more accurate than
$e$ and remaining in the same accuracy class $A$.

\medskip{}

The simplification consequent to the use of the second notion of big-O
is due to the following theorem, which states that also the second
big-O formally behaves as expected:
\begin{thm}
\label{thm:2ndBigOh-prop}Under the assumptions of Def. \ref{def:2ndBigOh},
the following properties of $O_{\mathcal{J}}$, as $\eps\in\mathbb{I}$,
hold:
\begin{enumerate}[%
leftmargin=*,label=(\roman*),align=left %
]
\item \foreignlanguage{british}{\label{enu:2ndBigOh-rifl}$x_{\eps}=O_{\mathcal{J}}(x_{\eps})$ }
\item \label{enu:2ndBigOh-trans}$x_{\eps}=O_{\mathcal{J}}(y_{\eps})$ and
$y_{\eps}=O_{\mathcal{J}}(z_{\eps})$, then $x_{\eps}=O_{\mathcal{J}}(z_{\eps})$ 
\item \label{enu:2ndBigOh-prod}$O_{\mathcal{J}}(x_{\eps})\cdot O_{\mathcal{J}}(y_{\eps})=O_{\mathcal{J}}(x_{\eps}\cdot y_{\eps})$ 
\item \label{enu:2ndBigOh-sum}$O_{\mathcal{J}}(x_{\eps})+O_{\mathcal{J}}(y_{\eps})=O_{\mathcal{J}}\left(\left|x_{\eps}\right|+\left|y_{\eps}\right|\right)$ 
\item \label{enu:2ndBigOh-prodExt}$x_{\ep}\cdot O_{\mathcal{J}}(y_{\eps})=O_{\mathcal{J}}(x_{\eps}\cdot y_{\eps})$ 
\item \label{enu:2ndBigOh-sumEqualSummand}$O_{\mathcal{J}}(x_{\eps})+O_{\mathcal{J}}(x_{\eps})=O_{\mathcal{J}}(x_{\eps})$ 
\item \label{enu:2ndBigOh-sumExt}If $x_{\eps},y_{\eps}\ge0$ for all $\eps\in I$,
then $x_{\eps}+O_{\mathcal{J}}(y_{\eps})=O_{\mathcal{J}}(x_{\eps}+y_{\eps})$ 
\item \label{enu:2ndBigOh-prodScal1}$\forall k\in\R:\ O_{\mathcal{J}}(k\cdot x_{\eps})=O_{\mathcal{J}}(x_{\eps})$ 
\item \label{enu:2ndBigOh-prodScal2}$\forall k\in\R:\ k\cdot O_{\mathcal{J}}(x_{\eps})=O_{\mathcal{J}}(x_{\eps})$ 
\end{enumerate}
\end{thm}
\begin{proof}
\ref{enu:2ndBigOh-rifl}: There exists $A\in\mathcal{J}$ since $\mathcal{J}$
is non empty by assumption. For all $a\in A$ the property $x_{\eps}=O_{a,A}(x_{\eps})$
follows by \ref{enu:1stBigOh-rifl} of Thm. \ref{thm:usualBigOh-prop}.

\ref{enu:2ndBigOh-prod}: Once again, we prove this property to illustrate
the general idea of the proof of other properties like \ref{enu:2ndBigOh-trans},
\ref{enu:2ndBigOh-sum} and \ref{enu:2ndBigOh-sumEqualSummand}.

As usual, we have to prove that 
\[
x'_{\eps}=O_{\mathcal{J}}(x_{\eps})\ ,\ y'_{\eps}=O_{\mathcal{J}}(y_{\eps})\ \Rightarrow\ x'_{\eps}\cdot y'_{\eps}=O_{\mathcal{J}}(x_{\eps}\cdot y_{\eps}).
\]
Therefore, we assume 
\begin{align}
\exists A & \in\mathcal{J}\,\forall c\in A:\ x'_{\eps}=O_{c,A}(x_{\eps})\label{eq:2ndBigOh-x'-x}\\
\exists B & \in\mathcal{J}\,\forall c\in B:\ y'_{\eps}=O_{c,B}(y_{\eps}).\label{eq:2ndBigOh-y'-y}
\end{align}
The assumptions on $\mathcal{J}$ yield the existence of $C\in\mathcal{J}$
such that $C\subseteq A$ and $C\subseteq B$. Property \ref{enu:1stBigOh-subsetA}
of Thm. \ref{thm:usualBigOh-prop} and \eqref{eq:2ndBigOh-x'-x},
\eqref{eq:2ndBigOh-y'-y} give $x'_{\eps}=O_{c,C}(x_{\eps})$ and
$y'_{\eps}=O_{c,C}(y_{\eps})$ for all $c\in C$. We can thus apply
the analogous property \ref{enu:1stBigOh-prod} of Thm. \ref{thm:usualBigOh-prop}
to get the conclusion 
\[
x'_{\eps}\cdot y'_{\eps}=O_{c,C}(x_{\eps}\cdot y_{\eps})\qquad\forall c\in C.
\]
For the remaining properties \ref{enu:2ndBigOh-prodExt}, \ref{enu:2ndBigOh-sumExt},
\ref{enu:2ndBigOh-prodScal1} and \ref{enu:2ndBigOh-prodScal2} we
don't even need to use the assumptions on $\mathcal{J}$.
\end{proof}
The following result is a direct consequence of Ex. \ref{exa:stdAndNS},
\ref{exa:usualBigOh} and Def. \ref{def:2ndBigOh}. Its aim is not,
of course, to simplify but to show the unifying capability of the
notions of set of indices in connection with the results about $\gse$
and $\gsd$ we will show in the subsequent sections.
\begin{cor}
\label{cor:specialAlgebra}Let $\Omega\subseteq\R^{n}$ be an open
set and $(u_{\eps})\in\mathcal{C}^{\infty}(\Omega,\R)$ be a net of
smooth functions. We use the notations of \cite{GKOS} for moderate
and negligible nets related to the special algebra $\gss(\Omega)$,
and the notations of \cite{ToVe08} for similar notions related to
the algebra $\hat{\gs}(\Omega)$ of asymptotic functions. Then
\begin{enumerate}[%
leftmargin=*,label=(\roman*),align=left %
]
\item \foreignlanguage{british}{$(u_{\eps})\in\mathcal{E}_{M}^{\srm}(\Omega)$
if and only if 
\[
\forall K\Subset\Omega\,\forall\alpha\in\N^{n}\,\exists N\in\N:\ \sup_{x\in K}\left|\partial^{\alpha}u_{\eps}(x)\right|=O(\eps^{-N})\text{ as }\eps\in\mathbb{I}^{\srm}
\]
}\lyxdeleted{}{Fri Jul  4 15:14:47 2014}{\foreignlanguage{british}{
}} 
\item $(u_{\eps})\in\mathcal{N}^{\srm}(\Omega)$ if and only if 
\[
\forall K\Subset\Omega\,\forall\alpha\in\N^{n}\,\forall m\in\N:\ \sup_{x\in K}\left|\partial^{\alpha}u_{\eps}(x)\right|=O(\eps^{m})\text{ as }\eps\in\mathbb{I}^{\srm}
\]

\item $(u_{\eps})\in\mathcal{M}\left(\mathcal{E}(\Omega)^{\D_{0}}\right)$
if and only if 
\[
\forall K\Subset\Omega\,\forall\alpha\in\N^{n}\,\exists N\in\N:\ \sup_{x\in K}\left|\partial^{\alpha}u_{\eps}(x)\right|=O(\underline{\eps}^{-N})\text{ as }\eps\in\hat{\mathbb{I}}
\]

\item $(u_{\eps})\in\mathcal{N}\left(\mathcal{E}(\Omega)^{\D_{0}}\right)$
if and only if 
\[
\forall K\Subset\Omega\,\forall\alpha\in\N^{n}\,\forall m\in\N:\ \sup_{x\in K}\left|\partial^{\alpha}u_{\eps}(x)\right|=O(\underline{\eps}^{m})\text{ as }\eps\in\hat{\mathbb{I}}.
\]

\end{enumerate}
\end{cor}
Henceforth, $\Omega$ will always denote an open subset of $\R^{n}$.

\section{\label{sec:The-full-algebra}The full algebra $\gse$}

The idea to define the correct set of indices for the full algebra
is that in the definition of $\gse$ we always use representatives
evaluated at $\eps\odot\phi$ and consider the asymptotics for $\eps\to0^{+}$
and fixed $\phi$.
\begin{defn}
\label{def:setOfIndicesFullAlgebra}We define:
\begin{enumerate}[%
leftmargin=*,label=(\roman*),align=left %
]
\item \foreignlanguage{british}{$I^{\erm}:=\mathcal{A}_{0}=\left\{ \phi\in\D(\R^{n})\mid\int\phi=1\right\} $} 
\item $\mathcal{I}^{\erm}:=\left\{ \mathcal{A}_{q}\mid q\in\N\right\} $,
where $\mathcal{A}_{q}$ is the set of all $\phi\in\mathcal{A}_{0}$
such that $\int x^{\alpha}\cdot\phi(x)\diff{x}=0$ for all $\alpha\in\N^{n}$
with $1\le|\alpha|\le q$. 
\item For $\eps$, $e\in I^{\erm}$, we define $\eps\le e$ iff there exist
$\phi\in\mathcal{A}_{0}$, $r_{\eps}$, $r_{e}\in\R_{>0}$ such that

\begin{enumerate}
\item $r_{\eps}\le r_{e}$ 
\item $\eps=r_{\eps}\odot\phi$ and $e=r_{e}\odot\phi$. 
\end{enumerate}

\noindent Equivalently, we can define $\eps\le e$ iff there exists
$r\in\R_{>0}$ such that $r\le1$ and $\eps=r\odot e$.

\item $\mathbb{I}^{\erm}:=(I^{\erm},\le,\mathcal{I}^{\erm})$ 
\end{enumerate}
\end{defn}
\noindent We firstly note that if $\phi\in\mathcal{A}_{q}$ then also
$r\odot\phi\in\mathcal{A}_{q}$ for all $r\in\R_{>0}$ and thus 
\[
e\in A\in\mathcal{I}^{\erm}\then(\emptyset,e]\subseteq A.
\]
 Therefore $A_{\le e}=(\emptyset,e]$ and so $O_{e,A}$ doesn't depend
on $A$, and we can simply write $x_{\eps}=O_{e}(y_{\eps})$ as $\eps\in\mathbb{I}^{\erm}$.

\noindent Secondly, we note that, contrary to the case of $\mathbb{I}^{\srm}$
and $\hat{\mathbb{I}}$, in this case we don't have $(\emptyset,\eps_{0}]\in\mathcal{I}^{\erm}$.
Moreover, $A\in\mathcal{I}^{\erm}$ is not downward directed by $<$.
\begin{thm}
\noindent \label{thm:propSetOfIndicesFull}\ 
\begin{enumerate}[%
leftmargin=*,label=(\roman*),align=left %
]
\item \foreignlanguage{british}{\label{enu:SoIFull-orderedSet}$(I^{\erm},\le)$
is an ordered set} 
\item \label{enu:SoIFull-totallyDown}$(\emptyset,e]$ is totally ordered
and downward directed by $<$ for all $e\in I^{\erm}$ 
\item \label{enu:SoIFull-equiv}Let $f,g\in\R^{I^{\erm}}$ and $\phi\in I^{\erm}$.
We use both the notation $f_{\eps}=f(\eps)$ for evaluating these
maps. Then it results that 
\[
f_{\eps}=O_{\phi}(g_{\eps})\text{ as }\eps\in\mathbb{I}^{\erm}
\]
 if and only if 
\[
f(r\odot\phi)=O\left[g(r\odot\phi)\right]\text{ as }r\to0^{+}.
\]

\end{enumerate}
\end{thm}
\begin{proof}
\ref{enu:SoIFull-orderedSet}: Reflexivity follows from $1\odot\eps=\eps$.
In order to prove transitivity, assume 
\[
\eta=s\odot\epsilon\ ,\ s\le1\ ,\ \eps=r\odot e\ ,\ r\le1.
\]
 Then $\eta=s\odot(r\odot e)=sr\odot e$ and $sr\le1$, so that $\eta\le e$.
To prove antisymmetry, assume that 
\[
\eps=r\odot e\ ,\ r\le1\ ,\ e=s\odot\eps\ ,s\le1.
\]
 Then $\eps=rs\odot\eps$, which implies $rs=1$ by Lem. \ref{lem:identityForOperators}
and hence $r=s=1$.

\noindent \ref{enu:SoIFull-totallyDown}: Assume that $\eps=r\odot e$,
$r\le1$ and $\eta=s\odot e$, $s\le1$. Therefore $\eps=\frac{r}{s}\odot\eta$
and we have $\eps\le\eta$ or $\eta\le\eps$ according to $\frac{r}{s}\le1$
or $\frac{s}{r}\le1$. Moreover, taking $t<\min(r,s)$ and $\sigma:=t\odot e\in(\emptyset,e]$
we get $\sigma<\eps$ and $\sigma<\eta$.

\noindent \ref{enu:SoIFull-equiv}: Assume that $f_{\eps}=O_{\phi}(g_{\eps})$,
i.e. 
\begin{equation}
\exists H\in\R_{>0}\,\exists\eps_{0}\le\phi\,\forall\ep\le\eps_{0}:\ \left|f_{\eps}\right|\le H\cdot\left|g_{\eps}\right|.\label{eq:SoIFull-Equiv-new}
\end{equation}
 Hence, we can write $\eps_{0}=r_{0}\odot\phi$ for some $0<r_{0}\le1$,
and for all $r\in(r,r_{0}]$ we get $\eps:=r\odot\phi\le\eps_{0}$.
Thus, condition \eqref{eq:SoIFull-Equiv-new} implies 
\[
\left|f_{\eps}\right|=\left|f\left(r\odot\phi\right)\right|\le H\cdot\left|g_{\eps}\right|=H\cdot\left|g\left(r\odot\phi\right)\right|,
\]
 which proves that $f(r\odot\phi)=O\left[g(r\odot\phi)\right]$ as
$r\to0^{+}$.

Vice versa, assume 
\begin{equation}
\forall r\in(0,r_{0}]:\ \left|f\left(r\odot\phi\right)\right|\le H\cdot\left|g\left(r\odot\phi\right)\right|,\label{eq:SoIFull-Equiv-old}
\end{equation}
 where $H,r_{0}\in\R_{>0}$. We can assume that $r_{0}\le1$, so that
setting $\eps_{0}:=r_{0}\odot\phi$ we have $\eps_{0}\le\phi$. For
each $\eps\le\eps_{0}$, we can write $\eps=r\odot\eps_{0}=r\cdot r_{0}\odot\phi$,
with $r\le1$. Thus $r\cdot r_{0}\in(0,r_{0}]$ and \eqref{eq:SoIFull-Equiv-old}
yields $\left|f\left(r\cdot r_{0}\odot\phi\right)\right|=\left|f_{\eps}\right|\le H\cdot\left|g\left(r\cdot r_{0}\odot\phi\right)\right|=H\cdot\left|g_{\eps}\right|$,
which is our conclusion.\end{proof}
\begin{cor}
\label{cor:I^e-SoI}$\mathbb{I}^{\erm}=(I^{\erm},\le,\mathcal{I}^{\erm})$
is a set of indices.
\end{cor}
The following natural result and the limit $\lim_{\eps\le e}\text{int}\left[\text{supp}(\eps)\right]=\emptyset$
justify our notation $(\emptyset,e]$.
\begin{cor}
\label{cor:I^e-iso}For all $e\in I^{\erm}$, the map 
\begin{align*}
\omega:(0,1] & \to(\emptyset,e]\\
r & \mapsto r\odot e
\end{align*}
is an isomorphism of ordered sets.\end{cor}
\begin{proof}
It is easy to prove that $\omega$ is order preserving and bijective.
It remains to prove that $\omega^{-1}$ is order preserving. Assume
$r\odot e\le s\odot e$, with $r,s\in(0,1]$. Hence $r\odot e=ts\odot e$
for some positive $t\le1$. Therefore, $\frac{ts}{r}=1$ and hence
$r=ts\le s$ as claimed.\end{proof}
\begin{cor}
\label{cor:fullAlgSeqToZero}Let $e\in I^{\erm}$ and $(z_{k})_{k\in\N}$
be a sequence in $(\emptyset,e]$. Let $(x_{\eps})$, $(y_{\eps})\in\R^{I^{\erm}}$
be two nets of real numbers defined in $I^{\erm}$. Then
\begin{enumerate}[%
leftmargin=*,label=(\roman*),align=left %
]
\item \foreignlanguage{british}{\label{enu:zeroSeqInFullAlg}$(z_{k})_{k\in\N}\to\emptyset$
in $(\emptyset,e]$ if and only if $\lim_{k\to+\infty}\underline{z_{k}}=0^{+}$} 
\item \label{enu:zeroSubSeqInFullAlg}From $(z_{k})_{k\in\N}$ we can always
extract a strictly decreasing subsequence which tends to $\emptyset$
in $(\emptyset,e]$ 
\item \label{enu:neg1stBigOhInFullAlg}The asymptotic relation $x_{\eps}=O_{e}(y_{\eps})$
as $\eps\in\mathbb{I}^{\erm}$ is false if and only if \foreignlanguage{english}{for
each $H\in\R_{>0}$ there exists a sequence $(\eps_{k})_{k\in\N}$
of $A_{\le e}$ such that:}

\begin{enumerate}[%
leftmargin=*,align=left %
]
\item \foreignlanguage{british}{$(\eps_{k})_{k\in\N}$ is strictly decreasing }
\item $(\eps_{k})_{k\in\N}\to\emptyset$ \foreignlanguage{english}{in $A_{\le e}$ }
\item $\forall k\in\N:\ \left|x_{\eps_{k}}\right|>H\cdot\left|y_{\eps_{k}}\right|$ 
\end{enumerate}
\end{enumerate}
\end{cor}
\begin{proof}
Property \ref{enu:zeroSeqInFullAlg} holds because we can write $z_{k}=r_{k}\odot e$
for a unique $r_{k}\in\R_{>0}$, and $(z_{k})_{k\in\N}$ tends to
$\emptyset$ if and only if $\lim_{k\to+\infty}r_{k}=0^{+}$, i.e.\ if
and only if $\lim_{k\to+\infty}\underline{z_{k}}=0^{+}$.

From \ref{enu:SoIFull-totallyDown} of Thm. \ref{thm:propSetOfIndicesFull},
Lem. \ref{lem:fromSeqToDecresSeq} and Thm. \ref{thm:neg-1stBigOh}
we directly obtain the proof of \ref{enu:zeroSubSeqInFullAlg} and
\ref{enu:neg1stBigOhInFullAlg}.
\end{proof}
For the sake of completeness, we recall the usual notations for $\gse(\Omega)$:
\begin{defn}
\label{def:fullCA}\ 
\begin{enumerate}[%
leftmargin=*,label=(\roman*),align=left %
]
\item \foreignlanguage{british}{$U(\Omega):=\left\{ (\phi,x)\in\mathcal{A}_{0}\times\Omega\mid\supp(\phi)\subseteq\Omega-x\right\} $;} 
\item We say that $R\in\mathcal{E}^{\erm}(\Omega)$ iff $R:U(\Omega)\longrightarrow\R$
and 
\[
\forall\phi\in\mathcal{A}_{0}:\ R(\phi,-)\text{ is smooth on }\Omega\cap\{x\in\R^{n}\mid\supp(\phi)\subseteq\Omega-x\};
\]

\item We say that $R\in\mathcal{E}_{M}^{\erm}(\Omega)$ iff $R\in\mathcal{E}^{\erm}(\Omega)$
and 
\[
\forall K\Subset\Omega\,\forall\alpha\in\N^{n}\,\exists N\in\N\,\forall\phi\in\mathcal{A}_{N}:\ \sup_{x\in K}\left|\partial^{\alpha}R(\eps\odot\phi,x)\right|=O(\eps^{-N});
\]

\item We say that $R\in\mathcal{N}^{\erm}(\Omega)$ iff $R\in\mathcal{E}^{\erm}(\Omega)$
and 
\[
\forall K\Subset\Omega\,\forall\alpha\in\N^{n}\,\forall m\in\N\,\exists q\in\N\,\forall\phi\in\mathcal{A}_{q}:\ \sup_{x\in K}\left|\partial^{\alpha}R(\eps\odot\phi,x)\right|=O(\eps^{m});
\]

\item $\gs^{\erm}(\Omega):=\mathcal{E}_{M}^{\erm}(\Omega)/\mathcal{N}^{\erm}(\Omega)$
is called the \emph{full Colombeau algebra.} 
\end{enumerate}
\end{defn}
\noindent We first give an equivalent characterization of $\mathcal{E}^{\erm}(\Omega)$
as follows: For $\phi\in\mathcal{A}_{0}$, set 
\[
\Omega_{\phi}:=\left\{ x\in\Omega\mid\text{supp}(\phi)\subseteq\Omega-x\right\} .
\]
Note that when $\Omega_{\phi}=\emptyset$, the set $\mathcal{C}^{\infty}(\Omega_{\phi},\R)$
has a single element. If $X$, $Y$ and $Z$ are sets and $f:X\times Y\longrightarrow Z$,
$g:X\longrightarrow Z^{Y}$ are maps, we set 
\begin{align*}
f^{\wedge}:x\in X & \mapsto f(x,-)\in Z^{Y}\\
g^{\vee}:(x,y)\in X\times Y & \mapsto g(x)(y)\in Z.
\end{align*}
The maps $(-)^{\wedge}$ and $(-)^{\vee}$ can be used to express
the property of Cartesian closedness of the category of sets (see
e.g.\ \cite{AHS}), i.e.\ $\left(Z^{Y}\right)^{X}\simeq Z^{X\times Y}$.

Since $R\in\mathcal{E}^{e}(\Omega)$ iff $R^{\wedge}:\mathcal{A}_{0}\longrightarrow\bigcup_{\phi\in\mathcal{A}_{0}}\mathcal{C}^{\infty}(\Omega_{\phi},\R)$
and $R(\phi,-)\in\mathcal{C}^{\infty}(\Omega_{\phi},\R)$ for all
$\phi\in\mathcal{A}_{0}$, $R^{\wedge}\in\prod_{\phi\in\mathcal{A}_{0}}\mathcal{C}^{\infty}(\Omega_{\phi},\R)$.
By Cartesian closedness of the category of sets: 
\begin{equation}
\mathcal{E}^{\erm}(\Omega)\simeq\prod_{\phi\in\mathcal{A}_{0}}\mathcal{C}^{\infty}(\Omega_{\phi},\R).\label{eq:E^e-iso}
\end{equation}
It is also possible to see \eqref{eq:E^e-iso} as a diffeomorphism
of diffeological spaces, see \cite{GiWu14}.

\noindent Since $R\mapsto R^{\wedge}$ is a bijection, we can equivalently
define the full algebra $\gse(\Omega)$ starting from $u\in\prod_{\phi\in\mathcal{A}_{0}}\mathcal{C}^{\infty}(\Omega_{\phi},\R)$
and considering 
\[
u^{\vee}:(\phi,x)\in U(\Omega)\mapsto u(\phi)(x)\in\R.
\]
This motivates the following
\begin{defn}
\noindent \label{def:P(Omega)}$\mathcal{P}^{\erm}(\Omega):=\prod_{\eps\in I^{\erm}}\mathcal{C}^{\infty}(\Omega_{\eps},\R)$.
\end{defn}
\noindent We can say that elements of $\mathcal{P}^{\erm}(\Omega)$
are $I^{\erm}$-indexed nets $(u_{\eps})$ such that $u_{\eps}\in\mathcal{C}^{\infty}(\Omega_{\eps},\R)$.
The following theorem represents the unifying and simplifying capabilities
of the notions of set of indices and its asymptotic relations. It
underscores, also for the full Colombeau algebra, the importance of
the logical structure $\forall K\,\forall\alpha\,\exists N$, $\forall K\,\forall\alpha\,\forall m$
and the use of an asymptotic relation as $\underline{\eps}\to0$.
\begin{thm}
\noindent \label{thm:fullEquiv}Let $u=(u_{\eps})\in\mathcal{P}^{\erm}(\Omega)$,
then 
\begin{enumerate}[%
leftmargin=*,label=(\roman*),align=left %
]
\item \foreignlanguage{british}{\label{enu:fullEquiv-moderate}$u^{\vee}\in\mathcal{E}_{M}^{\erm}(\Omega)$
if and only if 
\[
\forall K\Subset\Omega\,\forall\alpha\in\N^{n}\,\exists N\in\N:\ \sup_{x\in K}\left|\partial^{\alpha}u_{\eps}(x)\right|=O\left(\underline{\eps}^{-N}\right)\text{ as }\eps\in\mathbb{I}^{\erm}
\]
}\lyxdeleted{}{Fri Jul  4 15:14:47 2014}{\foreignlanguage{british}{
}} 
\item \label{enu:fullEquiv-negligible}$u^{\vee}\in\mathcal{N}^{\erm}(\Omega)$
if and only if 
\[
\forall K\Subset\Omega\,\forall\alpha\in\N^{n}\,\forall m\in\N:\ \sup_{x\in K}\left|\partial^{\alpha}u_{\eps}(x)\right|=O\left(\underline{\eps}^{m}\right)\text{ as }\eps\in\mathbb{I}^{\erm}
\]

\end{enumerate}
\end{thm}
\begin{proof}
\noindent \ref{enu:fullEquiv-negligible}: Fix $K\Subset\Omega$,
$\alpha\in\N^{n}$ and $m\in\N$. By Def. \ref{def:2ndBigOh} of $O=O_{\mathcal{I}^{\erm}}$,
the condition 
\begin{equation}
\sup_{x\in K}\left|\partial^{\alpha}u_{\eps}(x)\right|=O\left(\underline{\eps}^{m}\right)\label{eq:fullEquiv-supNegl}
\end{equation}
 means 
\[
\exists q\in\N\,\forall\phi\in\mathcal{A}_{q}:\ \sup_{x\in K}\left|\partial^{\alpha}u_{\eps}(x)\right|=O_{\phi}\left(\underline{\eps}^{m}\right).
\]
 By \ref{enu:SoIFull-equiv} of Thm. \ref{thm:propSetOfIndicesFull},
this is equivalent to 
\[
\exists q\in\N\,\forall\phi\in\mathcal{A}_{q}:\ \sup_{x\in K}\left|\partial^{\alpha}u^{\vee}(r\odot\phi,x)\right|=O\left[\left(\underline{r\odot\phi}\right)^{m}\right]\text{ as }r\to0^{+}.
\]
 But if $\phi\ne0$, $\underline{r\odot\phi}=\text{diam}\left(\text{supp}(r\odot\phi)\right)=r\cdot\text{diam}\left(\text{supp}(\phi)\right)=:r\cdot H_{\phi}$,
and the same equality holds also if $\phi=0$ if we set $H_{\phi}:=\frac{1}{r}$.
Thus $O\left[\left(\underline{r\odot\phi}\right)^{m}\right]=O\left(r^{m}\cdot H_{\phi}^{m}\right)=O\left(r^{m}\right)$.
Therefore, \eqref{eq:fullEquiv-supNegl} is equivalent to 
\[
\exists q\in\N\,\forall\phi\in\mathcal{A}_{q}:\ \sup_{x\in K}\left|\partial^{\alpha}u^{\vee}(r\odot\phi,x)\right|=O\left(r^{m}\right)\text{ as }r\to0^{+}
\]
 as claimed.

\ref{enu:fullEquiv-moderate}: Fix $K\Subset\Omega$ and $\alpha\in\N^{n}$.
We firstly need to reformulate the condition 
\begin{equation}
\exists N\in\N\,\forall\phi\in\mathcal{A}_{N}:\ \sup_{x\in K}\left|\partial^{\alpha}u^{\vee}(r\odot\phi,x)\right|=O\left(r^{-N}\right)\text{ as }r\to0^{+}\label{eq:fullEquiv-moder1}
\end{equation}
 so that the term $r^{-N}$ doesn't depend on $N$, which appears
also in $\forall\phi\in\mathcal{A}_{N}$. We can consider 
\begin{equation}
\exists N,q\in\N\,\forall\phi\in\mathcal{A}_{q}:\ \sup_{x\in K}\left|\partial^{\alpha}u^{\vee}(r\odot\phi,x)\right|=O\left(r^{-N}\right)\text{ as }r\to0^{+}.\label{eq:fullEquiv-moder2}
\end{equation}
 In fact, \eqref{eq:fullEquiv-moder1} implies \eqref{eq:fullEquiv-moder2}
for logical reasons (the former is a particular case of the latter,
the one where $q=N$). Vice versa, assuming \eqref{eq:fullEquiv-moder2},
we have 
\[
\forall\phi\in\mathcal{A}_{q}:\ \sup_{x\in K}\left|\partial^{\alpha}u^{\vee}(r\odot\phi,x)\right|=O\left(r^{-N}\right)\text{ as }r\to0^{+}
\]
 for some $q,N\in\N$. If $q\le N$, we get \eqref{eq:fullEquiv-moder1}
from $\mathcal{A}_{q}\supseteq\mathcal{A}_{N}$. If $q>N$, then $r^{-q}>r^{-N}$
for $0<r<1$, so $y_{r}=O\left(r^{-N}\right)$ implies $y_{r}=O\left(r^{-q}\right)$
and we obtain \eqref{eq:fullEquiv-moder1} once again. Thus, the following
part of \eqref{eq:fullEquiv-moder2} 
\[
\exists q\in\N\,\forall\phi\in\mathcal{A}_{q}:\ \sup_{x\in K}\left|\partial^{\alpha}u^{\vee}(r\odot\phi,x)\right|=O\left(r^{-N}\right)\text{ as }r\to0^{+},
\]
 as proved above, can be equivalently written as 
\[
\sup_{x\in K}\left|\partial^{\alpha}u_{\eps}(x)\right|=O\left(\underline{\eps}^{-N}\right)\text{ as }\eps\in\mathbb{I}^{\erm}.\qedhere
\]

\end{proof}

\section{Big-O for uniform asymptotic relations}

Frequently, the asymptotic relation 
\[
\sup_{x\in K}\left|\partial^{\alpha}u_{\eps}(x)\right|=O\left(\eps^{-N}\right)\text{ as }\eps\to0^{+}
\]
 is expressed by saying that: $\partial^{\alpha}u_{\eps}(x)=O\left(\eps^{-N}\right)$
uniformly for $x\in K$. With this we mean 
\[
\exists H\in\R_{>0}\,\exists\eps_{0}\in(0,1]\,\forall\eps\in(0,\eps_{0}]\,\forall x\in K:\ \left|\partial^{\alpha}u_{\eps}(x)\right|\le H\cdot\eps^{-N}.
\]
 In this section, we want to see that this is a general possibility
in every set of indices. On the one hand, this will permit a further
simplification in our formulas, but on the other hand it hides the
choice of the particular seminorm $f\mapsto\sup_{x\in K}\left|\partial^{\alpha}f(x)\right|$
we are considering in these algebras.
\begin{defn}
\label{def:unifAsymptRel}Let $\mathbb{I}=(I,\le,\mathcal{I})$ be
a set of indices, $K\subseteq\R^{n}$ and $(x_{\eps})$, $(y_{\eps}):I\longrightarrow\R^{K}$.
Let $a\in A\in\mathcal{I}$ and $\mathcal{J}\subseteq\mathcal{I}$
be a non empty subset of $\mathcal{I}$ such that 
\[
\forall A,B\in\mathcal{J}\,\exists C\in\mathcal{J}:\ C\subseteq A\cap B.
\]
 Then:
\begin{enumerate}[%
leftmargin=*,label=(\roman*),align=left %
]
\item \foreignlanguage{british}{We say that $x_{\eps}=O_{a,A}^{K}(y_{\eps})$
as $\eps\in\mathbb{I}$ if 
\[
\exists H\in\R_{>0}\,\exists\eps_{0}\in A_{\le a}\,\forall\eps\in A_{\le\eps_{0}}\,\forall x\in K:\ |x_{\eps}(x)|\le H\cdot|y_{\eps}(x)|.
\]
}
\item We say that $x_{\eps}=O_{\mathcal{J}}^{K}(y_{\eps})$ as $\eps\in\mathbb{I}$
if $\exists A\in\mathcal{J}\,\forall a\in A:\ x_{\eps}=O_{a,A}^{K}(y_{\eps})$.
As above, we simply write $x_{\eps}=O^{K}(y_{\eps})$ (as $\eps\in\mathbb{I}$)
if $\mathcal{J}=\mathcal{I}$.
\end{enumerate}
\end{defn}
\begin{thm}
\label{thm:unifBigOhProp}Under the assumptions of Def. \ref{def:unifAsymptRel},
for both $O_{a,A}^{K}$ and $O_{\mathcal{J}}^{K}$ all the properties
of Thm. \ref{thm:usualBigOh-prop} and Thm. \ref{thm:2ndBigOh-prop}
hold.\end{thm}
\begin{proof}
We can prove this result repeating the proofs of Thm. \ref{thm:usualBigOh-prop}
and Thm. \ref{thm:2ndBigOh-prop}, or noting that the order relation
on $\R^{K}$ given by 
\[
f\le g\DIff\forall x\in K:\ f(x)\le g(x)
\]
 inherits from the usual order relation on $\R$ all the properties
we need.\end{proof}
\begin{cor}
\label{cor:fullWithAsymptBigOh}Let $u=(u_{\eps})\in\mathcal{P}^{\erm}(\Omega)$,
then
\begin{enumerate}[%
leftmargin=*,label=(\roman*),align=left %
]
\item \foreignlanguage{british}{\label{enu:fullWithAsymptBigOh-moderate}$u^{\vee}\in\mathcal{E}_{M}^{\erm}(\Omega)$
if and only if 
\[
\forall K\Subset\Omega\,\forall\alpha\in\N^{n}\,\exists N\in\N:\ \partial^{\alpha}u_{\eps}=O^{K}\left(\underline{\eps}^{-N}\right)\text{ as }\eps\in\mathbb{I}^{\erm}
\]
}\lyxdeleted{}{Fri Jul  4 15:14:47 2014}{\foreignlanguage{british}{
}} 
\item \label{enu:fullWithAsymptBigOh-negligible}$u^{\vee}\in\mathcal{N}^{\erm}(\Omega)$
if and only if 
\[
\forall K\Subset\Omega\,\forall\alpha\in\N^{n}\,\forall m\in\N:\ \partial^{\alpha}u_{\eps}=O^{K}\left(\underline{\eps}^{m}\right)\text{ as }\eps\in\mathbb{I}^{\erm}
\]

\end{enumerate}
\end{cor}
Applying the uniform asymptotic relation to the algebra $\hat{\gs}$
of asymptotic functions, we obtain
\begin{cor}
\label{cor:NSAAsymptBigOh}Let $(u_{\eps})\in\mathcal{C}^{\infty}(\Omega,\R)$
be a net of smooth functions.
\begin{enumerate}[%
leftmargin=*,label=(\roman*),align=left %
]
\item \foreignlanguage{british}{$(u_{\eps})\in\mathcal{M}\left(\mathcal{E}(\Omega)^{\D_{0}}\right)$
if and only if 
\[
\forall K\Subset\Omega\,\forall\alpha\in\N^{n}\,\exists N\in\N:\ \partial^{\alpha}u_{\eps}=O^{K}(\underline{\eps}^{-N})\text{ as }\eps\in\hat{\mathbb{I}}
\]
}\lyxdeleted{}{Fri Jul  4 15:14:47 2014}{\foreignlanguage{british}{
}} 
\item $(u_{\eps})\in\mathcal{N}\left(\mathcal{E}(\Omega)^{\D_{0}}\right)$
if and only if 
\[
\forall K\Subset\Omega\,\forall\alpha\in\N^{n}\,\forall m\in\N:\ \partial^{\alpha}u_{\eps}=O^{K}(\underline{\eps}^{m})\text{ as }\eps\in\hat{\mathbb{I}}.
\]

\end{enumerate}
\end{cor}

\section{diffeomorphism invariant algebras}

We can use the notion of set of indices to simplify the definitions
of the diffeomorphism invariant algebra $\gsd$. We will see a simpler
formulation, but not a unifying one. On the contrary, this reformulation
underscores some conceptual differences between $\gss$, $\gse$,
$\hat{\gs}$ on the one hand and $\gsd$ on the other hand.

We start by recalling the following
\begin{defn}
\label{def:diffInvCA}\ 
\begin{enumerate}[%
leftmargin=*,label=(\roman*),align=left %
]
\item \foreignlanguage{british}{$\mathcal{E}^{C}(\Omega):=\mathcal{C}^{\infty}(U(\Omega),\R)$.} 
\item We say $\phi\in\mathcal{C}_{\brm}^{\infty}((0,1]\times\Omega,\mathcal{A}_{q})$
if and only if $\phi\in\mathcal{C}^{\infty}((0,1]\times\Omega,\mathcal{A}_{q})$
and 
\[
\forall K\Subset\Omega\,\exists B\subseteq\R^{n}\text{ bounded}\,\forall\alpha\in\N^{n}\,\forall\eps\in(0,1]\,\forall x\in K:\ \text{supp}\left[\partial^{\alpha}\phi(\eps,x)\right]\subseteq B.
\]

\item We say $R\in\mathcal{E}_{M}^{C}(\Omega)$ if and only if $R\in\mathcal{E}^{C}(\Omega)$
and 
\begin{multline*}
\forall K\Subset\Omega\,\forall\alpha\in\N^{n}\,\exists N\in\N\,\forall\phi\in\mathcal{C}_{\brm}^{\infty}((0,1]\times\Omega,\mathcal{A}_{0}):\\
\sup_{x\in K}\left|\partial_{x}^{\alpha}R\left(\eps\odot\phi(\eps,x),x\right)\right|=O\left(\eps^{-N}\right).
\end{multline*}

\item We say $R\in\mathcal{N}^{C}(\Omega)$ if and only if $R\in\mathcal{E}_{M}^{C}(\Omega)$
and 
\begin{multline*}
\forall K\Subset\Omega\,\forall\alpha\in\N^{n}\,\forall m\in\N\,\exists q\in\N\,\forall\phi\in\mathcal{C}_{\brm}^{\infty}((0,1]\times\Omega,\mathcal{A}_{q}):\\
\sup_{x\in K}\left|\partial_{x}^{\alpha}R\left(\eps\odot\phi(\eps,x),x\right)\right|=O\left(\eps^{m}\right).
\end{multline*}

\end{enumerate}
\end{defn}
\noindent There are three main problems in defining a set of indices
for $\gsd$: 
\begin{enumerate}[%
leftmargin=*,label=(\alph*),align=left %
]
\item \foreignlanguage{british}{\label{enu:1stProblG^d}Both the test functions
and the representatives are evaluated at $x\in\Omega$:}

\begin{itemize}
\item For $\gsd$ we have terms like: $\partial_{x}^{\alpha}R(\eps\odot\phi(\eps,x),x)$. 
\item For $\gss$, $\gse$ and $\hat{\gs}$ we have terms like: $\partial^{\alpha}u_{\eps}(x)$
(see Cor. \ref{cor:specialAlgebra}, Thm. \ref{thm:fullEquiv}, Cor.
\ref{cor:fullWithAsymptBigOh}, Cor. \ref{cor:NSAAsymptBigOh}). 
\end{itemize}
\item \label{enu:2ndProblG^d}In the definition of moderate representatives
for $\gsd$, test objects are taken in $\mathcal{C}_{\brm}^{\infty}((0,1]\times\Omega,\mathcal{A}_{0})$:
compare Def. \ref{def:fullCA} and Def. \ref{def:diffInvCA}. 
\item \label{enu:3rdProblG^d}The third problem is tied to the dependence
of $\phi\in\mathcal{C}_{\brm}^{\infty}((0,1]\times\Omega,\mathcal{A}_{q})$
on $\eps\in(0,1]$, and the use of the mollification 
\begin{equation}
\eps\odot\phi(\eps,x).\label{eq:mollifWithEps}
\end{equation}
 Of course, this is very different from the analogous $\eps\odot\phi$,
with $\phi\in\mathcal{A}_{q}$, used for $\gse$. For example, using
\eqref{eq:mollifWithEps} we cannot say that $\text{supp}(\eps\odot\phi)=\eps\cdot\text{supp}(\phi)$
and so $\underline{\eps\odot\phi}=\eps\cdot\underline{\phi}$. 
\end{enumerate}
Problem \ref{enu:1stProblG^d} is solved considering an isomorphic
version of $\mathcal{C}_{\brm}^{\infty}((0,1]\times\Omega,\mathcal{A}_{0})$:
\begin{defn}
\label{def:isoC_b}Let $q\in\N$, then we say $\phi\in\mathcal{C}_{\brm}^{\infty}((0,1],\mathcal{C}^{\infty}(\Omega,\mathcal{A}_{q}))$
if and only if the following conditions are satified: 
\begin{enumerate}[%
leftmargin=*,label=(\roman*),align=left %
]
\item \foreignlanguage{british}{$\phi\in\mathcal{C}^{\infty}((0,1],\mathcal{C}^{\infty}(\Omega,\mathcal{A}_{q}))$.
We use the notation $\phi=(\phi_{\eps})$ for this type of maps. }
\item $\forall K\Subset\Omega\,\exists B\subseteq\R^{n}\text{ bounded}\,\forall\alpha\in\N^{n}\,\forall\eps\in(0,1]\,\forall x\in K:\ \text{supp}\left[\partial^{\alpha}\phi_{\eps}(x)\right]\subseteq B$. 
\end{enumerate}
\end{defn}
\noindent We have the isomorphism 
\[
\mathcal{C}_{\brm}^{\infty}((0,1]\times\Omega,\mathcal{A}_{q})\simeq\mathcal{C}_{\brm}^{\infty}((0,1],\mathcal{C}^{\infty}(\Omega,\mathcal{A}_{q}))
\]
in the category of diffeological spaces when both spaces are viewed
as subspaces of the corresponding functional spaces $\mathcal{C}^{\infty}((0,1]\times\Omega,\mathcal{A}_{q})$
and $\mathcal{C}^{\infty}((0,1],\mathcal{C}^{\infty}(\Omega,\mathcal{A}_{q}))$
(see \cite{GiWu14} for the definition of diffeology on $\mathcal{A}_{q}$).

\noindent Problem \ref{enu:2ndProblG^d} is solved considering the
asymptotic relation $O_{\mathcal{J}}$ generated by $\mathcal{J}=\left\{ \mathcal{C}_{\brm}^{\infty}((0,1]\times\Omega,\mathcal{A}_{0})\right\} $.
Actually, this is the first time we really need the asymptotic relation
$O_{\mathcal{J}}$ with $\mathcal{J}\subset\mathcal{I}$. This implies
that we need to use $O_{\mathcal{J}}$ for the moderateness condition
and $O_{\mathcal{I}}$ for the negligibility condition, i.e.\ we
use two different asymptotic relations.

Problem \ref{enu:3rdProblG^d} is solved by keeping the information
of the fixed test function $\phi\in\mathcal{C}_{\brm}^{\infty}((0,1]\times\Omega,\mathcal{A}_{q})$
``in the index'', i.e.\ by considering indices of the form $\eps=(r,\phi)$.
\begin{defn}
\label{def:SoIG^d}\ 
\begin{enumerate}[%
leftmargin=*,label=(\roman*),align=left %
]
\item \foreignlanguage{british}{$I_{q}^{\drm}(\Omega):=(0,1]\times\mathcal{C}_{\brm}^{\infty}((0,1],\mathcal{C}^{\infty}(\Omega,\mathcal{A}_{q}))$
for all $q\in\N$.} 
\item $I^{\drm}(\Omega):=I_{0}^{\drm}(\Omega)$
\item $\mathcal{I}^{\drm}(\Omega):=\left\{ I_{q}^{\drm}(\Omega)\mid q\in\N\right\} $
\item \label{blah}For $\eps=(r,\phi)$, $e=(s,\psi)\in I^{\drm}(\Omega)$,
define $\eps\le e$ if and only if $\phi=\psi$ and $r\le s$
\item $\mathbb{I}^{\drm}(\Omega):=(I^{\drm}(\Omega),\le,\mathcal{I}^{\drm}(\Omega))$
\item $\forall\eps=(r,\phi)\in I^{\drm}(\Omega)\:\forall x\in\Omega:\ \text{ev}_{\eps}(x):=(r\odot\phi_{r}(x),x)$.
Therefore $\text{ev}_{\eps}\in\Coo(\Omega,U(\Omega))$.
\item \label{bluh}If $\eps=(r,\phi)\in I^{d}(\Omega)$, then $\underline{\eps}:=r$.
\item If $u\in\mathcal{E}^{C}(\Omega)$, then $u_{\eps}:=u\circ\text{ev}_{\eps}$
for all $\eps\in I^{\drm}(\Omega)$. Note that $u_{\eps}\in\Coo(\Omega,\R)$.
\end{enumerate}
\end{defn}
\begin{thm}
\label{thm:SoIG^d}$\mathbb{I}^{\drm}(\Omega)=(I^{\drm}(\Omega),\le,\mathcal{I}^{\drm}(\Omega))$
is a set of indices.\end{thm}
\begin{proof}
This is a direct consequence of the definitions.
\end{proof}
\noindent Because of \ref{blah} and \ref{bluh} one may call this
a \emph{trivial} set of indexes.
\begin{thm}
\label{thm:equivG^d}Let $u\in\mathcal{E}^{C}(\Omega)$, and $\mathcal{J}:=\left\{ I_{0}^{\drm}(\Omega)\right\} $,
then
\begin{enumerate}[%
leftmargin=*,label=(\roman*),align=left %
]
\item \foreignlanguage{british}{$u\in\mathcal{E}_{M}^{C}(\Omega)$ if and
only if 
\[
\forall K\Subset\Omega\,\forall\alpha\in\N^{n}\,\exists N\in\N:\ \sup_{x\in K}\left|\partial^{\alpha}u_{\eps}(x)\right|=O_{\mathcal{J}}\left(\underline{\eps}^{-N}\right)\text{ as }\eps\in\mathbb{I}^{\drm}
\]
}\lyxdeleted{}{Fri Jul  4 15:14:47 2014}{\foreignlanguage{british}{
}} 
\item If $u\in\mathcal{E}_{M}^{C}(\Omega)$, then $u\in\mathcal{N}^{C}(\Omega)$
if and only if 
\[
\forall K\Subset\Omega\,\forall\alpha\in\N^{n}\,\forall m\in\N:\ \sup_{x\in K}\left|\partial^{\alpha}u_{\eps}(x)\right|=O\left(\underline{\eps}^{m}\right)\text{ as }\eps\in\mathbb{I}^{\drm}
\]

\end{enumerate}
\end{thm}
\begin{proof}
This is immediate from $u\left[\text{ev}_{(r,\varphi)}(x)\right]=u(r\odot\varphi_{r}(x),x)$
and $\underline{(r,\varphi)}=r$.
\end{proof}
\noindent We finally remark that, similarly to the case of $\cG^{\drm}$,
one can treat the algebras $\cG^{2}$ and $\hat{\cG}$ of \cite{GKOS}.

\section{\label{sec:Application-point-values}Application: point values characterization
of generalized function}

In this section we want to show that our unified point of view can
effectively be used to generalize proofs which hold for the special
algebra.

We assume:
\begin{enumerate}[%
leftmargin=*,label=(\roman*),align=left %
]
\item $\mathbb{I}=(I,\le,\mathcal{I})$ is a set of indices. All the big-O
relations in this section have to be meant as $\eps\in\mathbb{I}$.
\item \label{enu:assumpForSeq}For each $a\in A\in\mathcal{I}$ there exists
a sequence $(z_{k})_{k}$ of $A_{\le a}$ such that $(z_{k})_{k}\to\emptyset$
in $A_{\le a}$. Moreover $\forall b,c\in A_{\le a}:\ b<c$ or $c\le b$.
\item $\mathcal{J}\subseteq\mathcal{I}$ is a non empty subset of $\mathcal{I}$
such that $\forall A,B\in\mathcal{J}\,\exists C\in\mathcal{J}:\ C\subseteq A\cap B$.
\item There is a map $I \to (0,1]$, $\eps \mapsto \underline{\eps}$ such that
\[
\exists A\in\mathcal{I}\,\forall a\in A:\ \lim_{\eps\in A_{\le a}}\underline{\eps}=0.
\]

\item Let a map $\eps \mapsto \Omega_\eps$ be given, where $\Omega_\eps$ is an open subset of $\R^n$ for each $\eps \in I$. Then we set $\mathcal{E}^\mathbb{I}(\Omega) := \{ u: \bigcup_{\eps \in I} \{ \eps \} \times \Omega_\eps \to \R\ |\ u(\eps, \cdot) \in C^\infty(\Omega_\eps, \R)\ \forall \eps \in I \}$. We write $u_\eps$ instead of $u(\eps, \cdot)$.
We furthermore have two subsets $\mathcal{E}^\mathbb{I}_{M}(\Omega)$, $\mathcal{N}^\mathbb{I}(\Omega)$ of $\mathcal{E}^\mathbb{I}(\Omega)$
characterized as follows:

\begin{enumerate}
\item $u\in\mathcal{E}_{M}^{\mathbb{I}}(\Omega)$ iff \foreignlanguage{british}{
\[
\forall K\Subset\Omega\,\forall\alpha\in\N^{n}\,\exists N\in\N:\ \sup_{x\in K}\left|\partial^{\alpha}u_{\eps}(x)\right|=O_{\mathcal{J}}\left(\underline{\eps}^{-N}\right)
\]
}
\item If $u\in\mathcal{E}_{M}^{\mathbb{I}}(\Omega)$, then $u\in\mathcal{N}^{\mathbb{I}}(\Omega)$
iff
\[
\forall K\Subset\Omega\,\forall m\in\N:\ \sup_{x\in K}\left|u_{\eps}(x)\right|=O\left(\underline{\eps}^{m}\right).
\]

\end{enumerate}
\end{enumerate}
%
% \noindent To show that these assumptions can be interpreted in $\mathcal{G}^{\srm}$,
% $\hat{\mathcal{G}}$, $\mathcal{G}^{\erm}$, $\mathcal{G}^{\drm}$,
% $\mathcal{G}^{2}$, after giving a clear meaning to the map $\eps\mapsto\Omega_{\eps}$,
% we need the following

For $\mathbb{I}$ we consider the cases $\mathbb{I}^{\srm}$, $\hat{\mathbb{I}}$ and $\mathbb{I}^{\erm}$ for which $\Omega_\eps$ is given by $\Omega$, $\Omega$ and $\{x \in \R^n\ |\ \supp \varphi + x \subseteq \Omega\}$, respectively. Then we have the following formulation of Thm. 1.2.3 of \cite{GKOS}.

\begin{thm}
\label{thm:1.2.3}If $\mathbb{I}\in\{\mathbb{I}^{\srm},\mbox{\ensuremath{\hat{\mathbb{I}}}},\mathbb{I}^{\erm}\}$
and $(u_{\eps})\in\mathcal{E}_{M}^{\mathbb{I}}(\Omega)$, then the
following are equivalent:
\begin{enumerate}[%
leftmargin=*,label=(\roman*),align=left %
]
\item \foreignlanguage{british}{$u\in\mathcal{N}^{\mathbb{I}}(\Omega)$}
\item $\forall K\Subset\Omega\,\forall\alpha\in\N^{n}\,\forall m\in\N:\ \sup_{x\in K}\left|\partial^{\alpha}u_{\eps}(x)\right|=O\left(\underline{\eps}^{m}\right)\text{ as }\eps\in\mathbb{I}$.
\end{enumerate}
\end{thm}
\begin{proof}
The proof is only a reformulation of Thm. 1.2.3 of \cite{GKOS}, provided
we use $\underline{\eps}$ instead of $\eps$. We only note that in
one of the final steps we have
\[
\partial_{i}u_{\eps}(x)=\underbrace{\left(u_{\eps}(x+\underline{\eps}^{m+N}e_{i})-u_{\eps}(x)\right)}_{O(\underline{\eps}^{2m+N})}\underline{\eps}^{-m-N}-\underbrace{\frac{1}{2}\partial_{i}^{2}u_{\eps}(x_{\theta})}_{O(\underline{\eps}^{-N})}\underline{\eps}^{m+N}
\]
and these two big-O are both of the same type $O_{\mathcal{I}}$ if
$\mathbb{I}\in\{\mathbb{I}^{\srm},\mbox{\ensuremath{\hat{\mathbb{I}}}},\mathbb{I}^{\erm}\}$.
On the contrary, in the case of the diffeomorphism invariant algebras
$\mathcal{G}^{\drm}$, $\mathcal{G}^{2}$ the first big-O would be
$O_{\mathcal{I}}$, whereas the second one would be $O_{\mathcal{J}}$,
so the proof cannot be trivially generalized.
\end{proof}
\noindent We can now define
\begin{defn}
\label{def:Rtil}\ 
\begin{enumerate}[%
leftmargin=*,label=(\roman*),align=left %
]
\item $\mathcal{G}^{\mathbb{I}}(\Omega):=\mathcal{E}^\mathbb{I}_{M}(\Omega)/\mathcal{N}^\mathbb{I}(\Omega)$
\item $\Omega_{M}^{\mathbb{I}}:=\left\{ (x_{\eps})\in\Omega^{I}\mid\exists N\in\N:\ x_{\eps}=O_{\mathcal{J}}(\underline{\eps}^{-N})\right\} $
\item $(x_{\eps})\sim_{\mathbb{I}}(y_{\eps})$ iff $\forall m\in\N:\ x_{\eps}-y_{\eps}=O(\underline{\eps}^{m})$,
where $(x_{\eps})$, $(y_{\eps})\in\Omega_{M}^{\mathbb{I}}$
\item $\widetilde{\Omega}^{\mathbb{I}}:=\Omega_{M}^{\mathbb{I}}/\sim_{\mathbb{I}}$
\item If $\mathcal{P}(\eps)$ is a property of $\eps\in I$, then we write
\[
\forall^{\mathbb{I}}\eps:\ \mathcal{P}(\eps)
\]
iff $\exists A\in\mathcal{I}\,\forall a\in A\,\exists\eps_{0}\le a\,\forall\eps\in A_{\le\eps_{0}}:\ \mathcal{P}(\eps)$,
and we read it saying ``for $\eps\in\mathbb{I}$ sufficiently small
$\mathcal{P}(\eps)$ holds''.
\item $[x_{\eps}]\in\otilc^{\mathbb{I}}$ iff $ $$[x_{\eps}]\in\widetilde{\Omega}^{\mathbb{I}}$
and $\exists K\Subset\Omega\,\forall^{\mathbb{I}}\eps:\ x_{\eps}\in K$
\item If $u=[u_{\eps}]\in\mathcal{G}^{\mathbb{I}}(\Omega)$ and $ $$x\in\otilc^{\mathbb{I}}$,
then $u(x):=[u_{\eps}(x_{\eps})]$.
\end{enumerate}
\end{defn}
\noindent The following theorem is a simple generalization of Prop.
1.2.45 and Thm. 1.2.46 of \cite{GKOS} by applying assumption \ref{enu:assumpForSeq},
Lem. \ref{lem:fromSeqToDecresSeq} and Thm. \ref{thm:neg-1stBigOh}:
\begin{thm}
\label{thm:pointValuesChar}Let $u\in\mathcal{G}^{\mathbb{I}}(\Omega)$,
then:
\begin{enumerate}[%
leftmargin=*,label=(\roman*),align=left %
]
\item If $x\in\otilc^{\mathbb{I}}$, then $u(x)$ is a well-defined element
of $\Rtil^{\mathbb{I}}$.
\item $u=0$ in $\mathcal{G}^{\mathbb{I}}(\Omega)$ iff $u(x)=0$ in $\Rtil^{\mathbb{I}}$
for all $x\in\otilc^{\mathbb{I}}$.
\end{enumerate}
\end{thm}

\section{Conclusions}

The notions we introduced in this article are helpful for a unified
presentation of Colombeau algebras and highlight the conceptual analogies
between several Colombeau algebras. There are three ways to work with
the notions we introduced in this article. The first one is to work
in a generic set of indices. The second one is to work in the Colombeau
algebra we are interested in, but using the particular set of indices
that simplifies its definition. The third one has been presented in
section \ref{sec:Application-point-values}: we assume those properties
that hold in all the cases we are interested in. In our opinion, the
first one is the hardest because it introduces a great level of abstraction.
Of course, the results obtained using this abstract method are more
general because they apply to several different Colombeau type algebras.
At the same time, frequently these results are almost trivial generalizations
of analogous results already known for the special algebra.

\end{document}